\documentclass{amsart}
\usepackage[pagewise]{lineno}
\usepackage{amsthm}
\usepackage{amssymb}
\usepackage{mathtools}
\usepackage[all]{xy}
\usepackage{tikz}
\definecolor{lightgrey}{rgb}{.804,.804,.756}
\usepackage[pdfauthor={V. Lebed, L. Vendramin},
pdftitle={Cohomology and extensions of braces},
colorlinks=false,linkbordercolor=lightgrey,citebordercolor=lightgrey,urlbordercolor=lightgrey]{hyperref}
\usepackage{shuffle}

\newcommand{\Z}{\mathbb{Z}}

\newcommand{\Id}{\mathrm{Id}}

\newcommand{\Fun}{\operatorname{Fun}}
\newcommand{\Sym}{\operatorname{Sym}}
\newcommand{\Ext}{\operatorname{Ext}}
\newcommand{\CTExt}{\operatorname{CTExt}}
\newcommand{\Orb}{\operatorname{Orb}}
\newcommand{\Ker}{\operatorname{Ker}}
\renewcommand{\Im}{\operatorname{Im}}
\newcommand\Degen{{\scriptstyle\mathrm{D}}}
\newcommand\Norm{{\scriptstyle\mathrm{N}}}
\newcommand\RedC{RC}
\newcommand\RedZ{RZ}

\newcommand\RedH{RH}
\newcommand\RedQ{RQ}
\newcommand\CS{{\scriptstyle\mathrm{CS}}}

\theoremstyle{plain}
\newtheorem{thm}{Theorem}[section]

\newtheorem{pro}[thm]{Proposition}
\newtheorem{lem}[thm]{Lemma}

\theoremstyle{definition}
\newtheorem{defn}[thm]{Definition}
\newtheorem{rem}[thm]{Remark}
\newtheorem{example}[thm]{Example}

\newtheorem{notation}[thm]{Notation}

\usepackage{pifont}
\usepackage{enumitem}

\setlist{nolistsep}
    \AtBeginDocument{%
      \setlength\abovedisplayskip{2pt}
      \setlength\belowdisplayskip{2pt}}

\newdir{ >}{{}*!/-10pt/\dir{>}}

\numberwithin{equation}{section}

\begin{document}

\title[Cohomology of braces]{Cohomology and extensions of braces}

\author{Victoria Lebed}
\address{Laboratoire de Math\'ematiques Jean Leray, Universit\'e de Nantes, 2 rue de la Houssi\-ni\`ere, BP 92208 F-44322 Nantes Cedex 3, France}
\email{lebed.victoria@gmail.com, Victoria.Lebed@univ-nantes.fr}

\author{Leandro Vendramin}
\address{Depto. de Matem\'atica, FCEN,
Universidad de Buenos Aires, Pabell\'on I, Ciudad Universitaria (1428)
Buenos Aires, Argentina}
\email{lvendramin@dm.uba.ar}

\thanks{The work of L.~V. is partially supported by CONICET, PICT-2014-1376,
MATH-AmSud, and ICTP. V.~L. thanks the program ANR-11-LABX-0020-01 and Henri
Lebesgue Center (University of Nantes) for support. The authors are grateful to the reviewer for useful remarks and interesting suggestions for a further development of the subject.}



\begin{abstract}
Braces and linear cycle sets are algebraic structures playing a major role in the classification of involutive set-theoretic solutions to the Yang-Baxter equation. This paper introduces two versions of their (co)homology theories.	These theories mix the Harrison (co)homology for the abelian group structure and the (co)homology theory for general cycle sets, developed earlier by the authors. Different classes of brace extensions are completely classified in terms of second cohomology groups.
\end{abstract}

\maketitle

\section{Introduction}\label{S:Intro}

A \emph{(left) brace} is an abelian group $(A,+)$ with an additional group
operation~$\circ$ such that for all $a,b,c\in A$, the following
compatibility condition holds:
\begin{linenomath*}
	\begin{align}\label{E:Brace}
	a\circ (b+c)+a=a\circ b+a\circ c.
\end{align}
\end{linenomath*}
The two group structures necessarily share the same neutral element, denoted
by~$0$.  Braces, in a slightly different but equivalent form, were introduced
by Rump~\cite{MR2278047}; the definition above goes back to Ced{\'o},
Jespers, and Okni{\'n}ski~\cite{MR3177933}. To get a feeling of what
braces look like, and to convince oneself that they are not as rare in practice as one might
think, the reader is referred to Bachiller's classification of braces of
order~$p^3$ \cite{BachillerP3}. 
The growing interest into these structures is due to a number of reasons. 
First, braces generalize radical rings. Second, Catino-Rizzo and Catino-Colazzo-Stefanelli \cite{CaRi,CaCoSr,CaCoSr2} unveiled the role of an $F$-linear version of this notion into the classification problem for regular subgroups of affine groups over a field~$F$. Third, braces are 
enriched cycle sets, and are therefore important in the study of set-theoretic
solutions to the Yang-Baxter equation, as we now recall. 

A \emph{cycle set}, as defined by Rump~\cite{MR2132760}, is a set $X$ with a binary operation~$\cdot$ having bijective left translations $X \to X$, $a \mapsto b \cdot a$, 
and satisfying the relation
\begin{linenomath*}
	\begin{align}\label{E:Cyclic}
	(a\cdot b)\cdot (a\cdot c)=(b\cdot a)\cdot (b\cdot c).
\end{align}
\end{linenomath*}
Rump showed that \emph{non-degenerate} cycle sets (i.e., with invertible
squaring map $a \mapsto a \cdot a$) are in bijection with non-degenerate
involutive set-theoretic solutions to the Yang-Baxter equation. Such solutions
form a combinatorially rich class of structures, connected with many other
domains of algebra: semigroups of $I$-type, Bieberbach groups, Hopf algebras,
Garside groups, etc. The cycle set approach turned out to be extremely fruitful
for elucidating the structure of such solutions and obtaining classification
results (see, for instance,
{\cite{MR1637256,JesOknI,MR2278047,MR2442072,MR2383056,MR2652212,RumpClassical,
Chouraqui,MR3177933,MR3374524,GI15,MR3437282,2015arXiv150900420S,
Smok}}
and references therein).  In spite of the intensive ongoing research on cycle
sets, their structure is still far from being completely understood. This can
be illustrated by numerous conjectures and open questions in the area, many of
which were formulated by Gateva-Ivanova and Cameron~\cite{MR2095675, MR2885602}
and by Ced{\'o}, Jespers, and del R{\'{\i}}o~\cite{MR2584610}.

Etingof, Schedler, and Soloviev~\cite{MR1722951} initiated the study of the
structure group of a solution to the YBE---and in particular of a cycle set.
These ideas were further explored in~\cite{MR1769723} and~\cite{MR1809284} 
for non-involutive solutions.
Concretely, the \emph{structure group} $G_{(X,\cdot)}$ of a cycle set
$(X,\cdot)$ is the free group on the set~$X$, modulo the relations $(a \cdot b) a =
(b \cdot a) b$ for all $a,b \in X$\footnote{Some authors prefer an alternative relation set $a (a \cdot b) = b (b \cdot a)$, which defines an isomorphic group.}. 
In~\cite{MR1722951}, the structure group of
a non-degenerate cycle set $(X,\cdot)$ was shown to be isomorphic, as a set, to
the free abelian group $\Z^{(X)}$ on~$X$; see also~\cite{LV} for an explicit
graphical form of this isomorphism. The group $G_{(X,\cdot)}$ thus carries a
second, abelian, group structure---the one pushed back from $\Z^{(X)}$---and
becomes a brace. Moreover, $G_{(X,\cdot)}$ inherits a cycle set structure
from~$X$, and yields a key example of the following notion. A \emph{linear
cycle set} is a cycle set $(A,\cdot)$ with an abelian group operation~$+$
satisfying, for all $a,b,c\in A$, the compatibility conditions
\begin{linenomath*}
	\begin{align}
	a\cdot (b+c) &=a\cdot b+a\cdot c,\label{E:LinCyclic}\\
	(a+b)\cdot c &=(a\cdot b)\cdot (a\cdot c).\label{E:LinCyclic2}
\end{align}
\end{linenomath*}
This structure also goes back to Rump~\cite{MR2278047}, who showed it to be equivalent to the brace structure, via the relation $a\cdot b=a^{-1}\circ(a+b)$. 

Understanding structure groups and certain classes of their quotients is often regarded as a reasonable first step towards understanding cycle sets. Even better: Bachiller, Ced{\'o}, and Jespers~\cite{BCJ_Brace} recently reduced the classification problem for cycle sets to that for braces. This explains the growing interest towards braces and linear cycle sets. As pointed out by Bachiller, Ced{\'o}, Jespers, and Okni{\'n}ski~\cite{BCJO_family}, an extension theory for braces would be crucial for classification purposes, as well as for elaborating new examples. This served as motivation for our paper.

A cohomology theory for general cycle sets was developed by the authors in~\cite{LV}. The second cohomology groups were given particular attention: they were shown to encode central cycle set extensions. Here we propose homology and cohomology theories for linear cycle sets, and thus for braces. As usual, central linear cycle set extensions turn out to be classified by the second cohomology groups. 

For pedagogical reasons, we first study extensions that are trivial on the level of abelian groups, together with a corresponding (co)homology theory (Sections \ref{S:Hom}-\ref{S:Ext}). Such extensions are still of interest, since it is often the cycle set operation that is the most significant part of the linear cycle set structure (as in the example of structure groups). On the other hand, they are technically much easier to handle than the general extensions (Sections \ref{S:HomFull}-\ref{S:ExtFull}). We therefore found it instructive to present this ``reduced'' case before the general one.

When finishing this paper, we learned that an analogous extension theory was independently developed by Bachiller~\cite{BachillerSimple}, using the language of braces. 
Some fragments of it in the $F$-linear setting also appeared in the work of Catino-Colazzo-Stefanelli \cite{CaCoSr}.

 An alternative approach to extensions was suggested earlier by Ben David and Ginosar~\cite{BDG}. Concretely, they studied the lifting problem for bijective $1$-cocycles---which is yet another avatar of braces. Their work was translated into the language of braces by Bachiller~\cite{BachillerP3}. Our choice of the linear cycle set language leads to more transparent constructions. Moreover, it made possible the development of a full cohomology theory extending the degree~$2$ constructions motivated by the extension analysis. Such a theory was missing in all the previous approaches.

\section{Reduced linear cycle set cohomology}\label{S:Hom}

From now on we work with linear cycle sets (=LCS). As explained in the introduction, all constructions and results can be directly translated into the language of braces. We will perform this translation for major results only.

Take a LCS $(A, \cdot, +)$ and an abelian group $\Gamma$. For $n > 0$, let $\RedC_n(A;\Gamma)$ denote the abelian group 
$\Gamma \otimes_{\Z} \Z A^{\times n} \simeq \Gamma^{(A^{\times n})}$, modulo the linearity relation
\begin{linenomath*}
	\begin{align}\label{E:LinLast}
	\gamma (a_1,\ldots,a_{n-1}, a_n + a'_n) &= \gamma (a_1,\ldots,a_{n-1}, a_n) + \gamma (a_1,\ldots,a_{n-1}, a'_n)
\end{align}
\end{linenomath*}
for the last copy of~$A$. Denote by $\RedC^{\Degen}_n(A;\Gamma)$ the abelian subgroup of $\RedC_n(A;\Gamma)$ generated by the \emph{degenerate $n$-tuples}, i.e. $\gamma(a_1,\ldots,a_{n})$ with $a_i=0$ for some $1 \leqslant i \leqslant n$. Consider also the quotient $\RedC^{\Norm}_n(A;\Gamma) = \RedC_n(A;\Gamma)/\RedC^{\Degen}_n(A;\Gamma)$. Further, define the maps $\partial_{n} \colon \Gamma A^{\times n} \to \Gamma A^{\times (n-1)}$, $n > 1$, as the linearizations of
\begin{linenomath*}
	\begin{align}\label{E:LinCycleHom}
	\partial_{n}(a_1,\ldots,a_{n})= &(a_1 \cdot a_2,\ldots,a_1 \cdot a_{n})\notag\\	
	&+\sum\nolimits_{i=1}^{n-2}(-1)^{i} (a_1,\ldots,a_i+a_{i+1},\dots,a_{n})\\
	&+(-1)^{n-1}	(a_1,\ldots,a_{n-2},a_n).\notag
\end{align}
\end{linenomath*}
Complete this family of maps by $\partial_1=0$. Dually, for $n > 0$, let $\RedC^n(A;\Gamma)$ denote the set of maps $f \colon A^{\times n} \to \Gamma$ linear in the last coordinate:
\begin{linenomath*}\begin{align}\label{E:LinLastCo}
	f(a_1,\ldots,a_{n-1}, a_n + a'_n) &= f(a_1,\ldots,a_{n-1}, a_n) + f(a_1,\ldots,a_{n-1}, a'_n),
\end{align}\end{linenomath*}
and let $\RedC_{\Norm}^n(A;\Gamma) \subset \RedC^n(A;\Gamma)$ comprise the maps vanishing on all degenerate $n$-tuples. Define the maps $\partial^{n} \colon \Fun(A^{\times n},\Gamma) \to \Fun(A^{\times (n+1)},\Gamma)$, $n \geqslant 1$, by	
\begin{linenomath*}\begin{align}\label{E:LinCycleCoHom}
	(\partial^{n}f)(a_1,\ldots,a_{n+1})= & f(a_1 \cdot a_2,\ldots, a_1 \cdot a_{n+1})\notag\\	
	&+\sum\nolimits_{i=1}^{n-1}(-1)^{i} f(a_1,\ldots,a_i+a_{i+1},\dots,a_{n+1})\\
	&+(-1)^{n} f(a_1,\ldots, a_{n-1},a_{n+1}).\notag
\end{align}\end{linenomath*}	
These formulas resemble the group (co)homology construction for $(A,+)$. We will now show that they indeed define a (co)homology theory.

\begin{pro}\label{PR:LinCycleHom}
Let $(A, \cdot, +)$ be a linear cycle set and $\Gamma$ be an abelian group.
\begin{enumerate}
\item The maps $\partial_{\bullet}$ above
\begin{itemize}
\item square to zero: $\partial_{n-1}\partial_{n}=0$ for all $n>1$;
\item induce maps $\RedC_n(A;\Gamma) \to \RedC_{n-1}(A;\Gamma)$;
\item and further restrict to maps $\RedC^{\Degen}_n(A;\Gamma) \to \RedC^{\Degen}_{n-1}(A;\Gamma)$.
\end{itemize}
\item The maps $\partial^{\bullet}$ above 
\begin{itemize}
\item square to zero: $\partial^{n+1}\partial^{n}=0$ for all $n \geqslant 1$;
\item restrict to maps $\RedC^n(A;\Gamma) \to \RedC^{n+1}(A;\Gamma)$;
\item and further restrict to maps $\RedC_{\Norm}^n(A;\Gamma) \to \RedC_{\Norm}^{n+1}(A;\Gamma)$.
\end{itemize}
\end{enumerate}
\end{pro}

The induced or restricted maps from the proposition will be abusively denoted by the same symbols $\partial_{\bullet}$,\, $\partial^{\bullet}$.
In the proof we shall need the special properties of the zero element of a LCS.
\begin{lem}\label{L:LinCycleSet0}
	In any LCS~$A$, the relations $a\cdot 0=0$, $0\cdot a=a$ hold for all $a \in A$.
\end{lem}
\begin{proof}
		By the LCS axioms, one has $a\cdot 0=a\cdot
		(0+0)=a\cdot 0+a\cdot 0$ and hence $a\cdot 0=0$. Similarly, $0\cdot
		a=(0+0)\cdot a=(0\cdot 0)\cdot (0\cdot a)=0\cdot (0\cdot a)$, and the relation $a = 0\cdot a$ follows by cancelling out~$0$ (recall that the left translation $0 \cdot $ is bijective).
\end{proof}

\begin{proof}[Proof of Proposition~\ref{PR:LinCycleHom}]
	We treat only the homological statements here; they imply the
	cohomological ones by duality.

	The maps~$\partial_{n}$ can be presented as signed sums $\partial_{n} =
	\sum_{i=0}^{n-1} (-1)^i \partial_{n;i}$, where 
	\begin{linenomath*}\begin{align}
		\partial_{n;0}(a_1,\ldots,a_{n}) &= (a_1 \cdot a_2,\ldots,a_1 \cdot a_{n}),\label{E:LCSCub1}\\
		\partial_{n;i}(a_1,\ldots,a_{n}) &=  (a_1,\ldots,a_i+a_{i+1},\dots,a_{n}), \qquad 1 \leqslant i \leqslant n-2,\label{E:LCSCub2}\\
		\partial_{n;n-1}(a_1,\ldots,a_{n}) &= (a_1,\ldots,a_{n-2},a_n).\label{E:LCSCub3}
	\end{align}\end{linenomath*}
	The relation $\partial_{n-1}\partial_{n}=0$ then classically reduces to the
	``almost commutativity'' 
	$\partial_{n-1;j}\partial_{n;i}=\partial_{n-1;i}\partial_{n;j+1}$ for all
	$i \leqslant j$. In the case $i > 0$ this latter relation is either
	tautological, or follows from the associativity of~$+$. For $i = 0 <j$, it
	follows from the left distributivity~\eqref{E:LinCyclic} for~$A$. 
	For $i = 0 = j$, it is a consequence of the second LCS
	relation~\eqref{E:LinCyclic2} for~$A$. 

	Further, using the linearity~\eqref{E:LinCyclic} of the left translations $a_{n} \mapsto a_1 \cdot a_{n}$, one sees that when applied to
	expressions of type \[
	(\ldots,a_{n-1}, a_n + a'_n) - (\ldots,a_{n-1}, a_n) -
	(\ldots,a_{n-1}, a'_n),
	\]
	all the maps~$\partial_{n;i}$ yield expressions of the
	same type. Hence their signed sums~$\partial_{\bullet}$ induce a differential
	on $\RedC_{\bullet}$. The possibility to further restrict to
	$\RedC^{\Degen}_{\bullet}$ is guaranteed by 
	Lemma~\ref{L:LinCycleSet0}.
\end{proof}

Proposition~\ref{PR:LinCycleHom} legitimizes the following definition:

\begin{defn}\label{D:CycleSetHom}
	The \emph{reduced (normalized) cycles} / \emph{boundaries} / \emph{homology
	groups of a linear cycle set $(A,\cdot,+)$} with coefficients in an abelian
	group~$\Gamma$ are the cycles / boundaries / homology groups of the chain
	complex $(\RedC_{\bullet}(A;\Gamma), \partial_{\bullet})$ (respectively,
	$(\RedC^{\Norm}_{\bullet}(A;\Gamma), \partial_{\bullet})$) above. Dually,
	the \emph{reduced (normalized) cocycles} / \emph{coboundaries} /
	\emph{cohomology groups of $(A,\cdot,+)$} are those of the complex
	$(\RedC^{\bullet}(A;\Gamma), \partial^{\bullet})$ (respectively,
	$(\RedC_{\Norm}^{\bullet}(A;\Gamma), \partial^{\bullet})$). We use the
	usual notations $\RedQ_n(A;\Gamma)$, $\RedQ^{\Norm}_n(A;\Gamma)$,
	$\RedQ^n(A;\Gamma)$, $\RedQ_{\Norm}^n(A;\Gamma)$, where $Q$ is one of
	the letters $Z$, $B$, or $H$.
\end{defn}

\begin{rem}
	In the proof of Proposition~\ref{PR:LinCycleHom} we actually showed that
	our (co)homo\-logy constructions can be refined into (co)simplicial ones.
\end{rem}

\begin{example}\label{EX:H1}
	Recall from the introduction that for a non-degenerate cycle set
	$(X,\cdot)$, the free abelian group $(\Z^{(X)},+)$ can be seen as a linear
	cycle set, with the cycle set operation induced from~$\cdot$. In this case
	$\RedC_1(\Z^{(X)};\Gamma)$ is simply the abelian group $\Gamma \otimes_{\Z}
	\Z^{(X)} = \Gamma^{(X)}$, and for $a_1,a_2 \in \Z^{(X)}$ one calculates  				 	$\partial_1 (a_1 , a_2) = a_1 \cdot a_2 - a_2$. Standard arguments
	from LCS theory then yield $$\RedH_1(\Z^{(X)};\Gamma)
	\cong \Gamma^{(\Orb(X))},$$ where $\Orb(X)$ is the set of \emph{orbits} of~$X$,
	i.e., classes for the equivalence relation generated by $a_1 \cdot a_2 \sim
	a_2$ for all $a_1,a_2 \in X$. Similarly, one calculates the first reduced
	cohomology group: $$\RedH^1(\Z^{(X)};\Gamma) \cong \Fun(\Orb(X),\Gamma).$$
\end{example}

We finish with a comparison between the (co)homology of a LCS $(A, \cdot, +)$ and the (co)homology of its underlying cycle set $(A, \cdot)$, as defined in~\cite{LV}. Recall that the homology $H^{\CS}_{n}(A;\Gamma)$ of $(A, \cdot)$ is computed by the complex $(\Gamma^{(A^{\times n})} ,\,\partial^{\CS}_{n})$, where
\begin{linenomath*}\begin{align*}
	\partial^{\CS}_{n}(a_1,\ldots,a_{n})= \sum\nolimits_{i=1}^{n-1}(-1)^{i-1} ((a_i &\cdot a_{1},\ldots,a_i \cdot a_{i-1},a_i \cdot a_{i+1},\ldots,a_i \cdot a_{n})\\
	-&(a_1,\ldots,a_{i-1},a_{i+1},\ldots,a_{n})).
\end{align*}\end{linenomath*}
Dually, the cohomology $H_{\CS}^{n}(A;\Gamma)$ of $(A, \cdot)$ is computed by the complex $(\Fun(A^{\times n},\Gamma) ,$ $\,\partial_{\CS}^{n})$, 
 with $\partial_{\CS}^{n} f = f \circ \partial^{\CS}_{n+1}$. Define the map $S_n \colon \Gamma^{(A^{\times n})} \to \RedC_n(A;\Gamma)$ as the composition of the anti-symmetrization map
$$ \gamma(a_1,\ldots,a_{n}) \mapsto \sum_{\sigma \in Sym_{n-1}} (-1)^{\sigma} \gamma(a_{\sigma(1)},\ldots,a_{\sigma(n-1)},a_{n})$$
(where $(-1)^{\sigma}$ is the sign of the permutation~$\sigma$) and the obvious projection $\Gamma^{(A^{\times n})} \twoheadrightarrow \RedC_n(A;\Gamma)$.

\begin{pro}\label{PR:LCSandCShom}
Let $(A, \cdot, +)$ be a linear cycle set and $\Gamma$ be an abelian group. The map~$S$ defined above yields a map of chain complexes 
$$S_n \colon \, (\Gamma^{(A^{\times n})} ,\,\partial^{\CS}_{n}) \, \to \, (\RedC_n(A;\Gamma), \, \partial_{n}).$$ 
\end{pro}

\begin{proof}
One has to compare the evaluations of the maps $\partial_{n} S_n$ and $S_{n-1} \partial^{\CS}_{n}$ on $\gamma(a_1,\ldots,a_{n})$. For this it is convenient to use the decomposition $\partial_{n} = \sum_{i=0}^{n-1} (-1)^i \partial_{n;i}$ from \eqref{E:LCSCub1}-\eqref{E:LCSCub3}. For $0 < i < n-1$, the map $\partial_{n;i}S_n$ is zero: in its evaluation, the terms $\pm \gamma(\ldots, a_{j}+a_{k}, \ldots)$ and $\mp \gamma(\ldots, a_{k}+a_{j}, \ldots)$ (with the sum at the $i$th position) cancel. A careful sign inspection yields 
\begin{linenomath*}\begin{align*}
\partial_{n;0}S_n (\gamma&(a_1,\ldots,a_{n}))= \\
&\sum\nolimits_{i=1}^{n-1}(-1)^{i-1} S_{n-1} (\gamma (a_i \cdot a_{1},\ldots,a_i \cdot a_{i-1},a_i \cdot a_{i+1},\ldots,a_i \cdot a_{n})),\\
\partial_{n;n-1}S_n (\gamma&(a_1,\ldots,a_{n}))=\\
&\sum\nolimits_{i=1}^{n-1}(-1)^{n-1-i} S_{n-1} (\gamma (a_{1},\ldots, a_{i-1}, a_{i+1},\ldots, a_{n})),
\end{align*}\end{linenomath*}
hence the maps $\partial_{n} S_n$ and $S_{n-1} \partial^{\CS}_{n}$ coincide.
\end{proof}

As a consequence, one obtains the dual map 
$$S^n \colon \, (\RedC^n(A;\Gamma), \, \partial^{n}) \,\to\, (\Fun(A^{\times n},\Gamma) ,\,\partial_{\CS}^{n})$$ 
of cochain complexes, and the induced maps in (co)homology.

\section{Cycle-type extensions vs. reduced $2$-cocycles}\label{S:Ext}

We now turn to a study of the \emph{reduced $2$-cocycles} of a linear cycle set
$(A, \cdot,+)$, i.e., maps $f\colon A\times A\to \Gamma$ (where $\Gamma$ is an abelian group) satisfying
\begin{linenomath*}\begin{align}
	f(a,b+c) &= f(a,b) + f(a,c), \label{E:2Cocycle}\\
	f(a+b,c) &= f(a \cdot b,a \cdot c) + f(a,c)\label{E:2Cocycle2}
\end{align}\end{linenomath*}
for all $a,b,c\in A$. The last relation, together with the commutativity of~$+$, yields
\begin{linenomath*}\begin{align}\label{E:CS2Cocycle}
	f(a \cdot b,a \cdot c) + f(a,c) &= f(b \cdot a,b \cdot c) + f(b,c),
\end{align}\end{linenomath*}
implying $\partial_{\CS}^{2}(f)=0$, so our~$f$ is necessarily a cocycle of the cycle set $(A, \cdot)$.

Among the reduced $2$-cocycles we distinguish the \emph{reduced $2$-coboundaries}
\begin{linenomath*}\begin{align*}
	\partial^1(\theta)(a,b) &= \theta(a \cdot b) - \theta(b),
\end{align*}\end{linenomath*} 
where the map $\theta \colon A \to \Gamma$ is linear.

\begin{example}
	Let $A$ and $\Gamma$ be abelian groups. Consider the \emph{trivial linear cycle
	set} structure $a\cdot_{tr} b=b$ over $A$. A map $f\colon A\times A\to\Gamma$ is a reduced $2$-cocycle of this LCS if and only if $f$ is a \emph{bicharacter}, in the sense of the bilinearity relations
	\[
		f(a+b,c)=f(a,c)+f(b,c),\quad
		f(a,b+c)=f(a,b)+f(a,c).
	\]
	The reduced $2$-coboundaries are all trivial in this case. Thus $\RedH^2(A;\Gamma)$ is the abelian group of bicharacters of~$A$ with values in~$\Gamma$. Observe that for the cycle set $(A,\cdot_{tr})$, all the differentials $\partial_{\CS}^{n}$ vanish. The second cohomology group $H_{\CS}^2(A;\Gamma)$ of this cycle set thus comprises all the maps $f\colon A\times A\to\Gamma$, and is strictly larger than $\RedH^2(A;\Gamma)$.
\end{example}

\begin{example}
	Let $A=\{0,1,2,3\}=\Z/4$ be the cyclic group of $4$
	elements written additively. Then $A$ is a brace with 
	\begin{linenomath*}\begin{align*}
	a \circ b &= a+b+2ab, & a^{-1} &= (2a-1)a.
	\end{align*}\end{linenomath*}
	The corresponding linear cycle set structure on~$A$ is given by the
	operation 
	\begin{linenomath*}\begin{align*}
	a \cdot b &= a^{-1} \circ (a+b) = (1+2a)b,
	\end{align*}\end{linenomath*}
	which is $b$ when one of $a$, $b$ is even, and $b+2$ otherwise. Take $\Gamma=\{0,1\} = \Z/2$. For a map $f\colon \Z/4 \times \Z/4 \to \Z/2$, relation~\eqref{E:2Cocycle} means that $f$ is of the form $f(a,b) = b \phi(a)$ (where the product is taken in~$\Z/2$, and $b$ is reduced modulo~$2$), for some $\phi \colon \Z/4 \to \Z/2$. Relation~\eqref{E:2Cocycle2} then translates as
	$$\phi(a+b) = \phi(b + 2ab) + \phi(a).$$
The substitution $b=0$ yields $\phi(0)=0$. Analyzing other values of $a$ and $b$, one sees that $\phi(1)$ and $\phi(3)$ can be chosen arbitrarily, and $\phi(2)$ has to equal $\phi(1)+\phi(3)$. The reduced $2$-coboundaries are again trivial: a linear map $\theta \colon \Z/4 \to \Z/2$ is necessarily of the form $\theta(a)=at$ for some constant $t \in \Z/2$, yielding 
$$\theta(a \cdot b)= (a \cdot b)t =  (1+2a)bt = bt = \theta(b)$$
(since $2a = 0$ in $\Z/2$). Summarizing, one gets 
$$\RedH^2(\Z/4; \Z/2) \,\simeq\, \Z/2 \times \Z/2.$$	
Let us now turn to the underlying cycle set $(\Z/4, \cdot)$. Playing with condition~\eqref{E:CS2Cocycle}, one verifies that its $2$-cocycles are maps $f\colon \Z/4 \times \Z/4 \to \Z/2$ verifying 3 linear relations
\begin{linenomath*}\begin{align*}
&f(0,1) + f(0,3) = 0, \\
&f(2,1) + f(2,3) = 0,\\
&f(1,1) + f(1,3) + f(3,1) + f(3,3) = 0.
\end{align*}\end{linenomath*}
Its only non-trivial $2$-coboundary is $f(a,b)=ab \mod 2$. This implies
\begin{linenomath*}\begin{align*}
Z_{\CS}^2(\Z/4; \Z/2) &\simeq (\Z/2)^{4\times 4 -3}=(\Z/2)^{13}, \\
H_{\CS}^2(\Z/4; \Z/2) &\simeq (\Z/2)^{12}.
\end{align*}\end{linenomath*}
\end{example}


We will now construct extensions of our LCS~$A$ by~$\Gamma$  out of
$2$-cocycles, show that any central cycle-type extension is isomorphic to one
of this type, and that reduced $2$-cocycles, modulo reduced $2$-coboundaries, classify such
extensions.  

\begin{lem}\label{L:ExtFromCocycle}
	Let $(A, \cdot, +)$ be a linear cycle set, $\Gamma$ be an abelian group,
	and $f\colon A\times A\to \Gamma$ be a map. Then the abelian group $\Gamma
	\oplus A$ with the operation $$(\gamma,a) \cdot (\gamma',a') =
	(\gamma'+f(a,a'),\, a\cdot a'), \qquad \gamma,\gamma' \in \Gamma,\; a,a' \in
	A$$ is a linear cycle set if and only if $f$ is a reduced $2$-cocycle,
	$f\in \RedZ^2(A;\Gamma)$. 
\end{lem}

\begin{notation}\label{N:Ext}
The LCS from the lemma is denoted by $\Gamma \oplus_f A$.
\end{notation}

\begin{proof}
	The left translation invertibility for $\Gamma \oplus_f A$ follows from to the
	same property for~$A$.  Properties~\eqref{E:LinCyclic} and~\eqref{E:LinCyclic2}
	are equivalent for $\Gamma \oplus_f A$ to,
	respectively, properties~\eqref{E:2Cocycle} and~\eqref{E:2Cocycle2} from the
	definition of a $2$-cocycle for~$f$. The cycle set property~$\eqref{E:Cyclic}$
	follows from~\eqref{E:LinCyclic2} and the commutativity of~$+$. 
\end{proof}

Lemma~\ref{L:ExtFromCocycle} and the correspondence between linear cycle sets and braces yield the following result.
\begin{lem}\label{L:BraceTransl}
	Let $(A, \circ, +)$ be a brace, $\Gamma$ be an abelian group, and $f\colon A\times A\to \Gamma$ be a map. Then the abelian group $\Gamma \oplus A$ with the product
$$(\gamma,a) \circ (\gamma',a')=(\gamma+\gamma'+f(a,a'),a \circ a'),\qquad \gamma,\gamma'\in\Gamma,\; a,a'\in A,$$
	is a brace if and only if for the corresponding linear cycle set $(A, \cdot, +)$, the map $\overline{f}(a,b) = f(a, a \cdot b)$ is a reduced $2$-cocycle. 
\end{lem}

Before introducing the notion of LCS extensions, we need some preliminary definitions.

\begin{defn}\label{D:Central}
A \emph{morphism} between linear cycle sets $A$ and $B$ is a map $\varphi \colon A\to B$ preserving the structure, i.e. for all $a,a'\in A$ one has $\varphi(a+a')=\varphi(a)+\varphi(a')$ and $\varphi(a \cdot a')=\varphi(a) \cdot \varphi(a')$. The \emph{kernel} of~$\varphi$ is defined by $\Ker \varphi = \varphi^{-1}(0)$. The notions of the \emph{image} $\Im \varphi = \varphi(A)$, of a \emph{short exact sequence} of linear cycle sets, and of \emph{linear cycle subsets}, are defined in the obvious way. A linear cycle subset~$A'$ of~$A$ is called \emph{central} if for all $a \in A$, $a' \in A'$, one has $a \cdot a' = a'$ and $a' \cdot a = a$.
\end{defn}

For a LSC morphism $\varphi \colon A\to B$, $\Ker \varphi$ and $\Im \varphi$ are clearly linear cycle subsets of~$A$ and~$B$ respectively. Lemma~\ref{L:LinCycleSet0} can be rephrased by stating that $\{0\}$ is a central linear cycle subset of~$A$.

\begin{defn}
A \emph{central cycle-type extension} of a linear cycle set $(A, \cdot, +)$ by an abelian group $\Gamma$ is the datum of a short exact sequence of linear cycle sets
\begin{linenomath*}\begin{align}\label{E:ShortExactSq}
&0 \to \Gamma \overset{\iota}{\to} E \overset{\pi}{\to} A \to 0,
\end{align}\end{linenomath*}
where $\Gamma$ is endowed with the {trivial} cycle set structure $\gamma \cdot \gamma' = \gamma'$, its image $\iota (\Gamma)$ is central in~$E$ (in the sense of Definition~\ref{D:Central}), and the short exact sequence of abelian groups underlying~\eqref{E:ShortExactSq} splits.
\end{defn}

The adjective \emph{cycle-type} refers here to the fact that our extensions are interesting on the level of the cycle set operation~$\cdot$ only, and trivial on the level of the additive operation~$+$, since we require the short exact sequences to linearly split. More general extensions---those taking into account the additive operation as well---are postponed until the next section. Cycle-type extensions are important, for example, for comparing the LCS structures on the structure group of a cycle set before and after a cycle set extension (cf. the introduction for more detail on structure groups, and~\cite{LV} for the cycle set extension theory).

The LCS $\Gamma \oplus_f A$ from Lemma~\ref{L:ExtFromCocycle} is an extension of~$A$ by~$\Gamma$ in the obvious way. We now show that this example is essentially exhaustive.

\begin{defn}\label{D:ExtEquiv}
Two central cycle-type LCS extensions $\Gamma \overset{\iota}{\rightarrowtail} E \overset{\pi}{\twoheadrightarrow} A$ and $\Gamma \overset{\iota'}{\rightarrowtail} E' \overset{\pi'}{\twoheadrightarrow} A$ are called \emph{equivalent} if there exists a LCS isomorphism $\varphi \colon E \to E'$ making the following diagram commute:
\begin{equation}\label{E:ExtEquiv}
\begin{gathered}
\xymatrix@!0 @R=0.8cm @C=2cm{
& E \ar@{->>}[dr]^{\pi} \ar[dd]^{\varphi}_{\sim} & \\
\Gamma\ar@{ >->}[ur]^{\iota} \ar@{ >->}[dr]^{\iota'} & & A  \\
& E' \ar@{->>}[ur]^{\pi'} &}
\end{gathered}
\end{equation}
The set of equivalence classes of central cycle-type extensions of~$A$ by~$\Gamma$ is denoted by $\CTExt(A,\Gamma)$.
\end{defn}

\begin{lem}\label{L:section->cocycle}
Let $\Gamma \overset{\iota}{\rightarrowtail} E \overset{\pi}{\twoheadrightarrow} A$ be a central cycle-type LCS extension, and $s \colon A \to E$ be a linear section of~$\pi$. Then the map
\begin{linenomath*}\begin{align*}
\widetilde{f} \colon A \times A &\to E,\\
(a,a') &\mapsto s(a) \cdot s(a') - s(a \cdot a')
\end{align*}\end{linenomath*}
takes values in $\iota(\Gamma)$ and defines a reduced cocycle $f \in \RedZ^2(A;\Gamma)$. Extensions~$E$ and $\Gamma \oplus_f A$ are equivalent. 
Furthermore, a cocycle~$f'$ obtained from another section~$s'$ of~$\pi$ is cohomologous to~$f$.
\end{lem}

\begin{proof}
The computation
$$\pi(\widetilde{f}(a,a')) = \pi s(a) \cdot \pi s(a') - \pi s(a \cdot a') = a \cdot a' - a \cdot a' = 0$$
yields $\Im \widetilde{f} \subseteq \Ker \pi = \Im \iota$ (by the definition of a short exact sequence).  Hence the map $f \colon A \times A \to \Gamma$ can be defined by the formula $f = \iota^{-1}\widetilde{f}$. It remains to check relations~\eqref{E:2Cocycle}-\eqref{E:2Cocycle2} for this map. The linearity of~$s$ and of the left translations $t_b \colon a \mapsto b \cdot a$ gives
\begin{linenomath*}\begin{align*}
	\widetilde{f}(a,b+c) &= s(a) \cdot s(b+c) - s(a \cdot (b+c)) = s(a) \cdot (s(b)+s(c)) - s(a \cdot b + a \cdot c)\\
	&= s(a) \cdot s(b) + s(a) \cdot s(c) - s(a \cdot b) - s(a \cdot c) =	\widetilde{f}(a,b) + \widetilde{f}(a,c).
\end{align*}\end{linenomath*}
hence $f(a,b+c) = f(a,b) + f(a,c)$, by the linearity of~$\iota$. Similarly, one has
\begin{linenomath*}\begin{align*}
	\widetilde{f}(a+b,c) &= s(a+b) \cdot s(c) - s((a+b) \cdot c) \\
	&= (s(a)+s(b)) \cdot s(c) - s((a \cdot b) \cdot (a \cdot c)) \\
	&= (s(a)\cdot s(b)) \cdot (s(a)\cdot s(c)) + \widetilde{f}(a \cdot b,a \cdot c) - s(a \cdot b) \cdot s(a \cdot c)\\
	&= \widetilde{f}(a \cdot b,a \cdot c) + (\widetilde{f}(a,b) + s(a \cdot b)) \cdot (s(a)\cdot s(c)) - s(a \cdot b) \cdot s(a \cdot c)\\
	&\overset{(1)}{=} \widetilde{f}(a \cdot b,a \cdot c) + s(a \cdot b) \cdot (s(a)\cdot s(c)) - s(a \cdot b) \cdot s(a \cdot c)\\
	&= \widetilde{f}(a \cdot b,a \cdot c) + s(a \cdot b) \cdot (s(a)\cdot s(c)- s(a \cdot c))\\
	&= \widetilde{f}(a \cdot b,a \cdot c) + s(a \cdot b) \cdot \widetilde{f}(a,c)\\	
	&\overset{(2)}{=} \widetilde{f}(a \cdot b,a \cdot c) + \widetilde{f}(a,c).
\end{align*}\end{linenomath*}
In~$(1)$ we got rid of $\widetilde{f}(a,b) \in \iota(\Gamma)$ since the centrality of $\iota(\Gamma)$ yields
$$(\widetilde{f}(a,b) + x) \cdot y = (\widetilde{f}(a,b) \cdot x) \cdot (\widetilde{f}(a,b) \cdot y) = x \cdot y$$
for all $x,y \in E$. This centrality was also used in~$(2)$. Relation $f(a+b,c) = f(a \cdot b,a \cdot c) + f(a,c)$ is now obtained from the corresponding relation for~$\widetilde{f}$ by applying $\iota^{-1}$.

We will next show that the linear map $\varphi \colon \Gamma \oplus_f A \to E$, $\gamma \oplus a \mapsto \iota(\gamma) + s(a)$ yields an equivalence of extensions. It is bijective, the inverse given by the map $x \mapsto \iota^{-1}(x-s\pi (x)) \oplus \pi(x)$ (this map is well defined since $x-s\pi (x) \in \Ker \pi = \Im \iota$). Let us check that $\varphi$ entwines the cycle set operations. One has
\begin{linenomath*}\begin{align*}
\varphi((\gamma \oplus a) &\cdot (\gamma' \oplus a')) = \varphi((\gamma' + f(a,a')) \oplus a \cdot a') = \iota(\gamma' + f(a,a')) + s(a \cdot a') \\
&= \iota(\gamma') + \widetilde{f}(a,a') + (s(a) \cdot s(a') - \widetilde{f}(a,a')) = \iota(\gamma') + s(a) \cdot s(a')\\
&= s(a) \cdot \iota(\gamma') + s(a) \cdot s(a') = s(a) \cdot (\iota(\gamma') + s(a')) \\
&= (\iota(\gamma) +s(a)) \cdot (\iota(\gamma') + s(a')) = \varphi(\gamma \oplus a) \cdot \varphi(\gamma' \oplus a').
\end{align*}\end{linenomath*} 
We used the centrality of~$\iota(\gamma')$ and~$\iota(\gamma)$. The commutativity of the diagram~\eqref{E:ExtEquiv} is obvious, and completes the proof.

Suppose now that the reduced cocycles~$f$ and~$f'$ are obtained from the sections~$s$ and~$s'$ respectively. Put $\widetilde{\theta} = s - s' \colon A \to E$. This is a linear map with its image contained in $\Ker \pi = \Im \iota$. Hence it defines a linear map $\theta \colon A \to \Gamma$. To show that~$f$ and~$f'$ are cohomologous, we establish the property $f' - f = \partial^1 \theta$. One computes
\begin{linenomath*}\begin{align*}
	(\widetilde{f} - \widetilde{f}')(a,a') &= \widetilde{f}(a,a') - \widetilde{f}'(a,a')\\
	&= s(a) \cdot s(a')-s(a \cdot a') - s'(a) \cdot s'(a')+s'(a \cdot a') \\
	&= s(a) \cdot s(a')- s(a) \cdot s'(a') - \widetilde{\theta}(a\cdot a')\\
	&= s(a) \cdot (s(a')- s'(a')) - \widetilde{\theta}(a\cdot a')\\	
	&= s(a) \cdot \widetilde{\theta}(a') - \widetilde{\theta}(a\cdot a') \overset{(1)}{=} \widetilde{\theta}(a') - \widetilde{\theta}(a\cdot a').
\end{align*}\end{linenomath*}
In~$(1)$ we used the centrality of $\widetilde{\theta}(a')$. The desired relation is obtained by applying $\iota^{-1}$.
\end{proof}

We now compare extensions constructed out of different $2$-cocycles.

\begin{lem}\label{L:ExtEquivVsCohom}
Let $(A, \cdot, +)$ be a linear cycle set, $\Gamma$ be an abelian group, and $f, f' \in \RedZ^2(A;\Gamma)$ be two reduced $2$-cocycles. The linear cycle set extensions $\Gamma \oplus_f A$ and $\Gamma \oplus_{f'} A$ are equivalent if and only if $f$ and $f'$ are cohomologous.
\end{lem}

\begin{proof}
Suppose that a linear map $\varphi \colon \Gamma \oplus_f A \to \Gamma \oplus_{f'} A$ provides an equivalence of extensions. The commutativity of the diagram~\eqref{E:ExtEquiv} forces it to be of the form $\varphi (\gamma \oplus a) = (\gamma + \theta(a)) \oplus a$ for some linear map $\theta \colon A \to \Gamma$. Further, one computes
\begin{linenomath*}\begin{align*}
\varphi((\gamma \oplus a) \cdot (\gamma' \oplus a')) &= \varphi((\gamma' + f(a,a')) \oplus a \cdot a') \\
&= (\gamma' + f(a,a') + \theta(a \cdot a')) \oplus a \cdot a',\\
\varphi(\gamma \oplus a) \cdot \varphi(\gamma' \oplus a') &= ((\gamma + \theta(a)) \oplus a) \cdot ((\gamma' + \theta(a')) \oplus a') \\
&= (\gamma' + \theta(a') + f'(a,a')) \oplus a \cdot a'.
\end{align*}\end{linenomath*}
Thus the map~$\varphi$ entwines the cycle set operations if and only if $f'-f$ is the coboundary $\partial^1 \theta$.

In the opposite direction, take cohomologous cocycles $f$ and $f'$. This means that the relation $f'-f =\partial^1 \theta$ holds for a linear map $\theta \colon A \to \Gamma$. Repeating the arguments above, one verifies that the map $\varphi (\gamma \oplus a) = (\gamma + \theta(a)) \oplus a$ is an equivalence of extensions $\Gamma \oplus_f A \to \Gamma \oplus_{f'} A$.
\end{proof}

Put together, the preceding lemmas yield:

\begin{thm}\label{T:Ext=H2}
Let $(A, \cdot, +)$ be a linear cycle set and $\Gamma$ be an abelian group. The construction from Lemma~\ref{L:section->cocycle} yields a bijective correspondence
$$\CTExt(A,\Gamma) \overset{1:1}{\longleftrightarrow} \RedH^2(A; \Gamma).$$
\end{thm} 

We finish this section by observing that in degree~$2$, the normalization brings nothing new to the reduced LCS cohomology theory:

\begin{pro}\label{PR:NormInDeg2}
All reduced $2$-cocycles of a linear cycle set $(A, \cdot, +)$ are normalized. Moreover, one has an isomorphism in cohomology: 
$$\RedH^2(A;\Gamma) \cong \RedH_{\Norm}^2(A;\Gamma).$$
\end{pro}

\begin{proof}
Putting $a=b=0$ in the defining relation~\eqref{E:2Cocycle2} for a reduced $2$-cocycle~$f$, and using the properties of the element~$0$ from Lemma~\ref{L:LinCycleSet0}, one gets $f(0,c)=0$ for all $c \in A$. Moreover, $f(c,0)=0$ by linearity. So~$f$ is normalized, hence the identification $\RedZ^2(A;\Gamma) = \RedZ_{\Norm}^2(A;\Gamma)$. In degree~$1$ the normalized and usual complexes coincide, yielding the desired cohomology group isomorphism in degree~$2$.
\end{proof}

\section{Full linear cycle set cohomology}\label{S:HomFull}

The previous section treated linear cycle set extensions of the form $\Gamma \oplus_f A$. They can be thought of as the direct product $\Gamma \oplus A$ of LCS with the cycle set operation~$\cdot$ deformed by~$f$. From now on we will handle a more general situation: the additive operation~$+$ on $\Gamma \oplus A$ will be deformed as well. Most proofs in this general case are analogous to but more technical than those from the previous sections.  

Take a linear cycle set $(A, \cdot, +)$ and an abelian group $\Gamma$. For $i \geqslant 0, j \geqslant 1$, let $ShC_{i,j}(A;\Gamma)$ be the abelian subgroup of $\Gamma^{(A^{\times (i+j)})}$, generated by the \emph{partial shuffles} 
\begin{linenomath*}\begin{align}\label{E:PartialShuffle}
&\sum_{\sigma \in Sh_{r,j-r}} (-1)^{\sigma} \gamma (a_1,\ldots,a_i,a_{i+\sigma^{-1}(1)},\ldots,a_{i+\sigma^{-1}(j)})
\end{align}\end{linenomath*}
taken for all $1 \leqslant r \leqslant j-1$, $a_k \in A$, $\gamma \in \Gamma$. Here $Sh_{r,j-r}$ is the subset of all the permutations~$\sigma$ of $j$~elements satisfying $\sigma(1) \leqslant \cdots \leqslant \sigma(r)$, $\sigma(r+1) \leqslant \cdots \leqslant \sigma(j)$. The term \emph{shuffle} is used when $i=0$. Put $C_{i,j}(A;\Gamma) = \Gamma^{(A^{\times (i+j)})} / ShC_{i,j}(A;\Gamma)$.

Recall the notations
\begin{linenomath*}\begin{align}
\partial_{n;0}(a_1,\ldots,a_{n}) &= (a_1 \cdot a_2,\ldots,a_1 \cdot a_{n}),\label{E:Bisimpl}\\
\partial_{n;i}(a_1,\ldots,a_{n}) &=  (a_1,\ldots,a_i+a_{i+1},\dots,a_{n}), \qquad 1 \leqslant i \leqslant n-1,\label{E:Bisimpl2}
\end{align}\end{linenomath*}
from the proof of Proposition~\ref{PR:LinCycleHom}, and consider the coordinate omitting maps
\begin{linenomath*}\begin{align}
\partial'_{n;i}(a_1,\ldots,a_{n}) &= (a_1,\ldots,a_{i-1},a_{i+1},\ldots,a_n), \qquad 1 \leqslant i \leqslant n.\label{E:Bisimpl3}
\end{align}\end{linenomath*}
Combine (the linearizations of) these maps into what we will show to be horizontal and vertical differentials of a bicomplex:
\begin{linenomath*}\begin{align}
\partial^h_{i,j} &= \partial_{i+j;0} +\sum\nolimits_{k=1}^{i-1}(-1)^{k}\partial_{i+j;k} + (-1)^{i}\partial'_{i+j;i}, \qquad i \geqslant 1, j \geqslant 1;\label{E:dh}\\
-\partial^v_{i,j} &= \partial'_{i+j;i+1} +\sum\nolimits_{k=1}^{j-1}(-1)^{k}\partial_{i+j;i+k} + (-1)^{j}\partial'_{i+j;i+j}, \qquad i \geqslant 0, j \geqslant 2.\label{E:dv}
\end{align}\end{linenomath*}
Here the empty sums are zero by convention. 
As before, $C^{\Degen}_{i,j}(A;\Gamma)$ denotes the abelian subgroup of $\Gamma^{(A^{\times (i+j)})}$ generated by the degenerate $(i+j)$-tuples, and $C^{\Norm}_{i,j}(A;\Gamma)$ is the quotient $\Gamma^{(A^{\times (i+j)})} /(C^{\Degen}_{i,j}(A;\Gamma) + ShC_{i,j}(A;\Gamma))$.

Dually, for $f \in  \Fun(A^{\times (i+j)},\Gamma)$, put $\partial_h^{i,j}f = f \circ \partial^h_{i+1,j}$ and $\partial_v^{i,j}f = f \circ \partial^v_{i,j+1}$, where $i \geqslant 0, j \geqslant 1$, and $f$~is extended to $\Z^{(A^{\times (i+j)})}$ by linearity. Let $C^{i,j}(A;\Gamma)$ be the abelian group of the maps $A^{\times (i+j)} \to \Gamma$ whose linearization vanishes on all the partial shuffles~\eqref{E:PartialShuffle} (with $\gamma$ omitted), and let $C_{\Norm}^{i,j}(A;\Gamma) \subseteq C^{i,j}(A;\Gamma)$ comprise the maps which are moreover zero on all the degenerate $(i+j)$-tuples.

We now assemble these data into a chain and a cochain bicomplex structures with normalization.

\begin{thm}\label{T:LinCycleHomFull}
Let $(A, \cdot, +)$ be a linear cycle set and $\Gamma$ be an abelian group.
\begin{enumerate}
\item The abelian groups $\Gamma^{(A^{\times (i+j)})}$, $i \geqslant 0, j \geqslant 1$, together with the linear maps $\partial^h_{i,j}$ and $\partial^v_{i,j}$ above, form a chain bicomplex. In other words, the following relations are satisfied:
\begin{linenomath*}\begin{align}
&\partial^h_{i-1,j}\partial^h_{i,j}=0, \qquad i \geqslant 2, j \geqslant 1;\label{E:BiComplexH}\\
&\partial^v_{i,j-1}\partial^v_{i,j}=0, \qquad i \geqslant 0, j \geqslant 3;\label{E:BiComplexV}\\
&\partial^h_{i,j-1}\partial^v_{i,j} = \partial^v_{i-1,j}\partial^h_{i,j}, \qquad i \geqslant 1, j \geqslant 2.\label{E:BiComplexHV}
\end{align}\end{linenomath*}
Moreover, these maps restrict to the subgroups $ShC_{i,j}(A;\Gamma)$ and $C^{\Degen}_{i,j}(A;\Gamma)$, and thus induce chain bicomplex structures on $C_{i,j}(A;\Gamma)$ and on $C^{\Norm}_{i,j}(A;\Gamma)$.
\item The linear maps $\partial_h^{i,j}$ and $\partial_v^{i,j}$ yield a cochain bicomplex structure for the abelian groups $\Fun(A^{\times (i+j)},\Gamma)$, $i \geqslant 0, j \geqslant 1$. This structure restricts to $C^{i,j}(A;\Gamma)$ and further to $C_{\Norm}^{i,j}(A;\Gamma)$.
\end{enumerate}
\end{thm}

We abusively denote the induced or restricted maps from the theorem by the same symbols $\partial^h_{\bullet}$, $\partial^v_{\bullet}$, etc.

The proof of the theorem relies on the following interpretation of our bicomplex. Its $j$th row is almost the complex from Proposition~\ref{PR:LinCycleHom}, with a slight modification: the last entry in an $n$-tuple, to which the $\partial_{n;i}$ with $i > 0$ did nothing and on which $\partial_{n;0}$ acted by a left translation $a_n \mapsto a_1 \cdot a_n$, is replaced with the $j$-tuple of last elements behaving in the same way. In the $i$th column, the first $i$ entries of~$A^{\times (i+\bullet)}$ are never affected, and on the remaining entries the vertical differentials $\partial^v_{i,\bullet}$ act as the differentials from Proposition~\ref{PR:LinCycleHom} computed for the trivial cycle set operation $a \cdot b = b$. Alternatively, the $i$th column can be seen as the Hochschild complex for $(A,+)$ with coefficients in $A^{\times i}$, on which~$A$ acts trivially on both sides. Modding out $ShC_{i,j}(A;\Gamma)$ means passing from the Hochschild to the Harrison complex in each column.

\begin{proof}
As usual, it suffices to treat only the homological statements.

Due to the observation preceding the proof, the horizontal relation~\eqref{E:BiComplexH} and the vertical relation~\eqref{E:BiComplexV} follow from Proposition~\ref{PR:LinCycleHom}. For the mixed relation~\eqref{E:BiComplexHV}, note that the horizontal and vertical differentials involved affect, respectively, the first~$i$ and the last~$j$ entries of an $(i+j)$-tuple, with the exception of the $\partial_{n;0}$ component of~$\partial^h$. However, this component also commutes with~$\partial^v$ because of the linearity (with respect to~$+$) of the left translation $a_1 \cdot $ involved.

Applying a left translation $a \cdot$ to each entry of a partial shuffle~\eqref{E:PartialShuffle}, one still gets a partial shuffle. Consequently, the horizontal differentials~$\partial^h$ restrict to $ShC_{i,j}(A;\Gamma)$. In order to show that the~$\partial^v$ restrict to $ShC_{i,j}(A;\Gamma)$ as well, it suffices to check that the expression
\begin{multline*}
\sum_{\sigma \in Sh_{r,j-r}} (-1)^{\sigma} (a_{\sigma^{-1}(2)},\ldots,a_{\sigma^{-1}(j)})\\
+\sum_{k=1}^{j-1}(-1)^{k} \sum_{\sigma \in Sh_{r,j-r}} (-1)^{\sigma} (a_{\sigma^{-1}(1)},\ldots,a_{\sigma^{-1}(k)}+a_{\sigma^{-1}(k+1)}, \ldots, a_{\sigma^{-1}(j)})\\
+(-1)^{j} \sum_{\sigma \in Sh_{r,j-r}} (-1)^{\sigma}  (a_{\sigma^{-1}(1)},\ldots,a_{\sigma^{-1}(j-1)})
\end{multline*}
is a linear combination of shuffles for all $j \geqslant 1$, $1 \leqslant r \leqslant j-1$, $a_k \in A$. Denote by $S_1$, $S_2$, and $S_3$ the three sums in the expression above. We also use the classical notation
\begin{linenomath*}\begin{align*}
\shuffle_{r,j-r}(a_1, \ldots, a_j) &= \sum\nolimits_{\sigma \in Sh_{r,j-r}} (-1)^{\sigma} (a_{\sigma^{-1}(1)},\ldots,a_{\sigma^{-1}(j)})
\end{align*}\end{linenomath*}
for shuffles, and the convention $\shuffle_{0,j} = \shuffle_{j,0} = \Id$. Recall also notations \eqref{E:Bisimpl2}-\eqref{E:Bisimpl3}. The sums~$S_i$ then rewrite as
\begin{linenomath*}\begin{align*}
S_1 &= \shuffle_{r-1,j-r} \partial'_{j;1} + (-1)^r \shuffle_{r,j-r-1} \partial'_{j;r+1},\\
S_3 &= (-1)^r\shuffle_{r-1,j-r}\partial'_{j;r} + (-1)^j \shuffle_{r,j-r-1}\partial'_{j;j},\\
S_2 &= \sum\nolimits_{k=1}^{r-1} (-1)^k  \shuffle_{r-1,j-r} \partial_{j;k} + \sum\nolimits_{k=r+1}^{j-1} (-1)^k \shuffle_{r,j-r-1} \partial_{j;k},
\end{align*}\end{linenomath*}
with empty sums declared to be zero. The decomposition for~$S_1$ follows from the analysis of the two possibilities for $\sigma^{-1}(1)$ with $\sigma \in Sh_{r,j-r}$, namely,  $\sigma^{-1}(1) = 1$ and $\sigma^{-1}(1) = r+1$. The decomposition for~$S_3$ corresponds to the dichotomy $\sigma^{-1}(j) = r$ or $\sigma^{-1}(j) = j$. In~$S_3$, the summands with $\sigma^{-1}(k) = u \leqslant r < v = \sigma^{-1}(k+1)$ and $\sigma^{-1}(k) = v, \sigma^{-1}(k+1) = u$ appear with opposite signs and can therefore be discarded. The remaining ones can be divided into two classes: those with $\sigma^{-1}(k) < \sigma^{-1}(k+1) \leqslant r$ and those with $r < \sigma^{-1}(k) < \sigma^{-1}(k+1)$, giving the decomposition above. Our~$S_i$ are thus signed sums of shuffles, with the exception of the cases $r \in \{1,j-1\}$. For $r=1$, the non-shuffle terms $\partial'_{j;1}$ and $-\partial'_{j;1}$ appear in~$S_1$ and~$S_3$ respectively; they annihilate each other in the total sum. The case $r=j-1$ is treated similarly.

The possibility to restrict all the~$\partial^h$ and~$\partial^v$ to $C^{\Degen}_{i,j}(A;\Gamma)$ is taken care of, as usual, by Lemma~\ref{L:LinCycleSet0}. As a consequence, one obtains a chain bicomplex structure on $C^{\Norm}_{i,j}(A;\Gamma)$.
\end{proof}

We are now in a position to define the full (co)homology of a linear cycle set:

\begin{defn}\label{D:CycleSetHomFull}
The \emph{(normalized) cycles} / \emph{boundaries} / \emph{homology groups of a linear cycle set $(A,\cdot,+)$} with coefficients in an abelian group~$\Gamma$ are the cycles / boundaries / homology groups of the  total chain complex 
$$(\, C_{n}(A;\Gamma) = \oplus_{i+j=n} C_{i,j}(A;\Gamma),\, \partial_n |_{C_{i,j}} =\partial_{i,j}^{h}+(-1)^i \partial_{i,j}^{v})$$ 
(respectively, $C^{\Norm}_{n}(A;\Gamma) = (\oplus_{i+j=n} C^{\Norm}_{i,j}(A;\Gamma), \partial_n)$) of the bicomplex above. Dually, the \emph{(normalized) cocycles} / \emph{coboundaries} / \emph{cohomology groups of $(A,\cdot,+)$} are those of the complex $(C^{n}(A;\Gamma) = \oplus_{i+j=n} C^{i,j}(A;\Gamma), \, \partial^n = \partial_{n+1}^*)$ (respectively, $(C_{\Norm}^{n}(A;\Gamma) =\oplus_{i+j=n} C_{\Norm}^{i,j}(A;\Gamma), \partial^n)$). We use the usual notations $Q_n(A;\Gamma)$ etc., where $Q$ is one of the letters $Z$, $B$, or $H$.
\end{defn}

\begin{rem}
In fact our (co)chain bicomplex constructions can be refined into bisimplicial ones.
\end{rem}

\begin{rem}
Instead of considering the total complex of our bicomplex, one could start by, say, computing the homology $H^v_{i,\bullet}$ of each column. The horizontal differentials then induce a chain complex structure on each row $H^v_{\bullet,j}$. Observe that the first row is precisely the complex from Proposition~\ref{PR:LinCycleHom}. Its homology is then the reduced homology of our linear cycle set.
\end{rem}

\section{General linear cycle set  extensions}\label{S:ExtFull}

Our next step is to describe what a $2$-cocycle looks like for the full version of linear cycle set cohomology theory. Such a $2$-cocycle consists of two components $f,g \colon A \times A \to \Gamma$, seen as elements of $C^{1,1}(A;\Gamma) = \Fun(A \times A, \Gamma)$ and $C^{0,2}(A;\Gamma)= \Sym(A \times A, \Gamma)$ respectively. Here $\Sym$ denotes the abelian group of \emph{symmetric} maps, in the sense of \begin{linenomath*}\begin{align}\label{E:2CocycleSym}
	&g(a,b)=g(b,a).
\end{align}\end{linenomath*}
These maps should satisfy three identities, one for each component of 
\[
C^{3}(A;\Gamma) = C^{2,1}(A;\Gamma) \oplus C^{1,2}(A;\Gamma)\oplus C^{0,3}(A;\Gamma).
\]
Explicitly, these identities read
\begin{linenomath*}\begin{align}
	&f(a+b,c) = f(a \cdot b,a \cdot c) + f(a,c),\label{E:2CocycleFull}\\
	&f(a,b+c) - f(a,b) - f(a,c) = g(a \cdot b, a \cdot c) - g(b,c),\label{E:2CocycleFull2}\\
	&g(a,b) + g(a+b,c) = g(b,c) + g(a,b+c).\label{E:2CocycleFull3}
\end{align}\end{linenomath*}
In particular, $f$ is a $2$-cocycle of the cycle set $(A,\cdot)$, and $g$ is a symmetric $2$-cocycle of the group $(A,+)$. 
The reduced cocycles are precisely those with $g=0$. Further, the $2$-coboundaries are couples of maps
\begin{linenomath*}\begin{align}
	f(a,b) &= \theta (a \cdot b) - \theta(b),\label{E:2CoboundFull}\\
	g(a,b) &= \theta (a + b) - \theta(a) - \theta(b)\label{E:2CoboundFull2}
\end{align}\end{linenomath*}
for some $\theta \colon A \to \Gamma$.

We next give some elementary properties of $2$-cocycles and $2$-coboundaries.

\begin{lem}\label{L:2CocycleProperties}
	Let $(f,g)$ be a $2$-cocycle of a linear cycle set $(A, \cdot, +)$ with
	coefficients in an abelian group~$\Gamma$.
	\begin{enumerate}
		\item For all $x \in A$, one has 
			\begin{linenomath*}\begin{align*}
				&f(0,x) = f(x,0) = 0,\\
				&g(0,x) = g(x,0) = g(0,0).
			\end{align*}\end{linenomath*}
		\item The $2$-cocycle $(f,g)$ is normalized if and only if one has $g(0,0) = 0$.
	\end{enumerate}
\end{lem}

\begin{proof}
	Let us prove the first claim. The relation $f(0,x) = 0$ follows from~\eqref{E:2CocycleFull} by choosing
			$a=0$. Similarly, the relation $f(x,0) = 0$ is~\eqref{E:2CocycleFull2}
			specialized at $b=c=0$. Substitutions $b=0$ and either $a=0$ or $c=0$
			in~\eqref{E:2CocycleFull3} yield the last relation. 
			Now the second claim directly follows from the previous point. 
\end{proof}

\begin{lem}\label{L:ExtFromCocycleFull}
	Let $(A, \cdot, +)$ be a linear cycle set, $\Gamma$ be an abelian group, and $f,g \colon A\times A\to \Gamma$ be two maps. Then the set $\Gamma \times A$ with the operations
\begin{linenomath*}\begin{align*}
(\gamma,a) + (\gamma',a') &= (\gamma + \gamma' + g(a,a'),\, a + a'),\\
(\gamma,a) \cdot (\gamma',a') &= (\gamma' + f(a,a'),\, a \cdot a')
\end{align*}\end{linenomath*} 
for $\gamma,\gamma' \in \Gamma$,\, $a,a' \in A$, is a linear cycle set if and only if $(f,g)$ is a $2$-cocycle, $(f,g) \in Z^2(A;\Gamma)$. 
\end{lem}

\begin{notation}\label{N:ExtFull}
The LCS from the lemma is denoted by $\Gamma \oplus_{f,g} A$.
\end{notation}

\begin{proof}
	The left translation invertibility for $\Gamma \oplus_{f,g} A$ follows from
	the same property for~$A$.  Properties~\eqref{E:LinCyclic}
	and~\eqref{E:LinCyclic2} for $\Gamma \oplus_{f,g} A$ are equivalent to,
	respectively, properties~\eqref{E:2CocycleFull2} and~\eqref{E:2CocycleFull}
	for~$(f,g)$. The associativity and the commutativity of~$+$ on $\Gamma
	\oplus_{f,g} A$ are encoded by property~\eqref{E:2CocycleFull3} for $\Gamma
	\oplus_{f,g} A$ and the symmetry of~$g$ respectively. Finally, if $(f,g)$
	is a $2$-cocycle, then Lemma~\ref{L:2CocycleProperties} implies that
	$(-g(0,0),0)$ is the zero element for $(\Gamma \oplus_{f,g} A, \, +)$, and
	the opposite of $(\gamma, a)$ is $(-\gamma - g(0,0) - g(a,-a), \, -a)$.
\end{proof}

As we did in Lemma~\ref{L:BraceTransl}, we now translate Lemma~\ref{L:ExtFromCocycleFull} into the language of	braces.

\begin{lem}\label{L:BraceTranslFull}
	Let $(A, \circ, +)$ be a brace, $\Gamma$ be an abelian group, and $f,g\colon A\times A\to \Gamma$ be two maps. Then the set $\Gamma \times A$ with the operations
	\begin{linenomath*}\begin{align*}
	(\gamma,a)+(\gamma',a')&=(\gamma+\gamma'+g(a,a'),a+a'),\\
	(\gamma,a) \circ (\gamma',a')&=(\gamma+\gamma'+f(a,a'),a \circ a')
	\end{align*}\end{linenomath*}
for $\gamma,\gamma'\in\Gamma$,\, $a,a'\in A$, is a brace if and only if for the corresponding linear cycle set $(A, \cdot, +)$, the maps
\begin{linenomath*}\begin{align}\label{E:BraceAndLCSCocycles}
\overline{f}(a,b) = -f(a, a \cdot b) + g(a,b)
\end{align}\end{linenomath*}
 and $g$ form a $2$-cocycle $(\overline{f},g)\in Z^2(A;\Gamma)$. 
\end{lem}

\begin{proof}
Recall the correspondence $a\cdot b=a^{-1}\circ(a+b)$ between the corresponding brace and LCS operations. It can also be rewritten as $a \circ b = a+ a \ast b$, where the map $a \mapsto a \ast b$ is the inverse of the left translation $a \mapsto a \cdot b$.

Now, given any $(\overline{f},g)\in Z^2(A;\Gamma)$, the formulas from Lemma~\ref{L:ExtFromCocycleFull} describe a LCS structure on $\Gamma \times A$. Its operation~$\ast$ reads
\begin{linenomath*}\begin{align*}
	(\gamma,a) \ast (\gamma',a')&=(\gamma'-\overline{f}(a, a \ast a'),a \ast a').
\end{align*}\end{linenomath*}
The operations
	\begin{linenomath*}\begin{align*}
	(\gamma,a)+(\gamma',a')&=(\gamma+\gamma'+g(a,a'),a+a'),\\
	(\gamma,a) \circ (\gamma',a')&= (\gamma,a) + (\gamma,a) \ast (\gamma',a')\\
	&= (\gamma+\gamma'- \overline{f}(a,a \ast a')+g(a,a \ast a'),a \circ a')
	\end{align*}\end{linenomath*}
then yield a brace structure on $\Gamma \times A$. These formulas have the desired form, with 
$$f(a,a') = - \overline{f}(a,a \ast a')+g(a,a \ast a'),$$
which, through the substitution $b = a \ast a'$, is equivalent to~\eqref{E:BraceAndLCSCocycles}.

Conversely, starting from a brace structure on $\Gamma \times A$ of the desired form, one sees that its associated LCS structure is as described in Lemma~\ref{L:ExtFromCocycleFull} with some $(\overline{f},g)\in Z^2(A;\Gamma)$. Repeating the argument above, one obtains the relation~\eqref{E:BraceAndLCSCocycles} connecting $f$, $\overline{f}$, and $g$.
\end{proof}

\begin{defn}
A \emph{central extension} of a linear cycle set $(A, \cdot, +)$ by an abelian group $\Gamma$ is the datum of a short exact sequence of linear cycle sets
\begin{linenomath*}
\begin{align}\label{E:ShortExactSqFull}
&0 \to \Gamma \overset{\iota}{\to} E \overset{\pi}{\to} A \to 0,
\end{align}
\end{linenomath*}
where $\Gamma$ is endowed with the trivial cycle set structure, and its image $\iota (\Gamma)$ is central in~$E$ (in the sense of Definition~\ref{D:Central}). The notion of \emph{equivalence} for central cycle-type LCS extensions (Definition~\ref{D:ExtEquiv}) transports verbatim to these general extensions. The set of equivalence classes of central extensions of~$A$ by~$\Gamma$ is denoted by $\Ext(A,\Gamma)$.
\end{defn} 

The LCS $\Gamma \oplus_{f,g} A$ from Lemma~\ref{L:ExtFromCocycleFull} is an extension of~$A$ by~$\Gamma$ in the obvious way. We now show that this example is essentially exhaustive.

\begin{lem}\label{L:section->cocycleFull}
Let $\Gamma \overset{\iota}{\rightarrowtail} E \overset{\pi}{\twoheadrightarrow} A$ be a central LCS extension, and $s \colon A \to E$ be a set-theoretic section of~$\pi$. 
\begin{enumerate}
\item The maps $\widetilde{f},\widetilde{g} \colon A \times A \to E$ defined by
\begin{linenomath*}
	\begin{align*}
\widetilde{f} \colon (a,a') &\mapsto s(a) \cdot s(a') - s(a \cdot a'),\\
\widetilde{g} \colon  (a,a') &\mapsto s(a) + s(a') - s(a + a')
\end{align*}
\end{linenomath*}
both take values in $\iota(\Gamma)$ and determine a cocycle $(f,g) \in Z^2(A;\Gamma)$. 
\item The cocycle above is normalized if and only if $s$ is such, in the sense of $s(0)=0$.
\item Extensions~$E$ and $\Gamma \oplus_{f,g} A$ are equivalent. 
\item A cocycle $(f',g')$ obtained from another section~$s'$ of~$\pi$ is cohomologous to $(f,g)$. If both cocycles are normalized, then they are cohomologous in the normalized sense.
\end{enumerate}  
\end{lem}

\begin{lem}\label{L:ExtEquivVsCohomFull}
Let $(A, \cdot, +)$ be a linear cycle set, $\Gamma$ be an abelian group, and $(f,g)$ and $(f',g') \in Z^2(A;\Gamma)$ be $2$-cocycles. The linear cycle set extensions $\Gamma \oplus_{f,g} A$ and $\Gamma \oplus_{f',g'} A$ are equivalent if and only if $(f,g)-(f',g')$ is a normalized $2$-coboundary.
\end{lem}

Recall that a normalized $2$-coboundary is a couple of maps of the form $\partial^1 \theta$, where the map $\theta \colon A \to \Gamma$ is normalized, in the sense of $\theta(0) = 0$.

The proof of these lemmas is technical but conceptually analogous to the proofs of Lemmas~\ref{L:section->cocycle} and~\ref{L:ExtEquivVsCohom}, and will therefore be omitted.

Put together, the preceding lemmas prove

\begin{thm}\label{T:Ext=H2Full}
Let $(A, \cdot, +)$ be a linear cycle set and $\Gamma$ be an abelian group. The construction from Lemma~\ref{L:section->cocycleFull} yields a bijective correspondence
$$Ext(A,\Gamma) \overset{1:1}{\longleftrightarrow} H_{\Norm}^2(A; \Gamma).$$
\end{thm} 

In other words, the central extensions of LCS (and thus of braces) are completely determined by their second normalized cohomology groups.

\bibliographystyle{abbrv}
\bibliography{refs}

\end{document}